\documentclass{gtpart}
\usepackage{amsmath,amssymb,amsthm,stmaryrd}
\usepackage[all]{xy}
\usepackage[usenames,dvipsnames]{xcolor}
\usepackage{tikz}
\usetikzlibrary{shapes,positioning,fit}
\usepackage{url}
\usepackage{hyperref}
\usepackage{enumerate}
\usepackage{tensor}
\usepackage{mathrsfs}
\usepackage{graphicx}
\usepackage{mathtools}
\usetikzlibrary{arrows}
\title{Morava $E$-homology of Bousfield-Kuhn functors\\on odd-dimensional 
spheres}
\author{Yifei Zhu}
\givenname{Yifei}
\surname{Zhu}
\email{zyf@umn.edu}
\urladdr{https://yifeizhu.github.io}


\newtheorem{thm}[equation]{Theorem}

\newtheorem{prop}[equation]{Proposition}

\theoremstyle{definition}

\theoremstyle{remark}
\newtheorem{rmk}[equation]{Remark}
\newtheorem{ex}[equation]{Example}

\def\co{\colon\thinspace}
\newcommand{\mb}[1]{\mathbb{#1}}

\newcommand{\Hom}{\ensuremath{{\rm Hom}}}

\newcommand{\Spf}{{\rm Spf\thinspace}}

\newcommand{\cF}{\overline {\mb F}}

\newcommand{\CC}{{\cal C}}

\newcommand{\CP}{{\cal P}}
\newcommand{\CR}{{\cal R}}

\newcommand{\DL}{Dyer-Lashof~}

\newcommand{\BC}{{\mb C}}

\newcommand{\BF}{{\mb F}}
\newcommand{\BG}{{\mb G}}

\newcommand{\BP}{{\mb P}}
\newcommand{\BQ}{{\mb Q}}
\newcommand{\BR}{{\mb R}}

\renewcommand{\SS}{{\bf S}}
\newcommand{\PP}{{\bf P}}
\newcommand{\BW}{{\mb W}}
\newcommand{\BZ}{{\mb Z}}

\newcommand{\md}{~~{\rm mod}~}
\newcommand{\ad}{\text{and}}

\newcommand{\id}{{\rm id}}
\newcommand{\op}{{\rm op}}

\newcommand{\Ext}{{\rm Ext}}

\newcommand{\nul}{{\rm nul}}

\newcommand{\coker}{{\rm coker}}
\newcommand{\A}{\alpha}
\newcommand{\B}{\beta}

\newcommand{\G}{\Gamma}

\renewcommand{\o}{\omega}

\newcommand{\T}{\tau}

\newcommand{\ce}{\coloneqq}
\newcommand{\lb}{\llbracket}
\newcommand{\rb}{\rrbracket}

\newcommand{\wt}[1]{\textcolor{white}{#1} \!~}

\newcommand{\ch}[2]{{#1 \choose #2}}

\makeatletter
\DeclareRobustCommand\widecheck[1]{{\mathpalette\@widecheck{#1}}}
\def\@widecheck#1#2{%
    \setbox\z@\hbox{\m@th$#1#2$}%
    \setbox\tw@\hbox{\m@th$#1%
       \widehat{%
          \vrule\@width\z@\@height\ht\z@
          \vrule\@height\z@\@width\wd\z@}$}%
    \dp\tw@-\ht\z@
    \@tempdima\ht\z@ \advance\@tempdima2\ht\tw@ \divide\@tempdima\thr@@
    \setbox\tw@\hbox{%
       \raise\@tempdima\hbox{\scalebox{1}[-1]{\lower\@tempdima\box
\tw@}}}%
    {\ooalign{\box\tw@ \cr \box\z@}}}
\makeatother

\numberwithin{equation}{section}
\renewcommand{\theequation}{\thesection.\arabic{equation}}

\begin{document}

\begin{abstract}
 As an application of Behrens and Rezk's spectral algebra model for unstable 
 $v_n$-periodic homotopy theory, we give explicit presentations for the 
 completed $E$-homology of the Bousfield-Kuhn functor on odd-dimensional spheres 
 at chromatic level $2$, and compare them to the level $1$ case.  The latter 
 reflects earlier work in the literature on $K$-theory localizations.  
\end{abstract}

\maketitle

\section{Introduction}

The rational homotopy theory of Quillen and Sullivan studies unstable homotopy 
types of topological spaces modulo torsion, or equivalently, after inverting 
primes.  Such homotopy types are computable by means of their {\em algebraic 
models}.  In particular, Quillen showed that there are equivalences of homotopy 
categories 
\[
 {\rm Ho}_\BQ({\rm Top_*})_2 \simeq {\rm Ho}_\BQ({\rm DGL})_1 \simeq 
 {\rm Ho}_\BQ({\rm DGC})_2  
\]
between simply-connected pointed topological spaces localized with respect to 
rational homotopy equivalences, connected differential graded Lie algebras over 
$\BQ$, and simply-connected differential graded cocommutative coalgebras over 
$\BQ$ \cite[Theorem I]{Quillen}.  

Let $p$ be a prime, $\BF_p$ be the field with $p$ elements, and $\cF_p$ be its 
algebraic closure.  Working prime by prime, one has $p$-adic analogues where 
equivalences detected through $H_*(-; \BQ)$ are replaced by those through 
$H_*(-; \BF_p)$.  Various algebraic models for $p$-adic homotopy types of spaces 
were developed \cite{Kriz, Goerss, Mandell}.  In the modern language of homotopy 
theory, these models are often formulated in terms of ``spectral'' algebra.  For 
example, Mandell's model is given by the functor that takes a connected 
$p$-complete nilpotent space $X$ of finite $p$-type to the $\cF_p$-cochains 
$H\cF_p^{\,X}$, where $H\cF_p^{\,X}$ denotes the function spectrum 
$F(\Sigma^\infty X, H\cF_p)$.  This spectrum is a commutative algebra over 
$H\cF_p$.  

More generally, through the prism of chromatic homotopy theory, Behrens and Rezk 
have established spectral algebra models for unstable {\em $v_n$-periodic} 
homotopy types \cite{BKTAQ} (cf.~\cite{AroneChing, Heuts, BKTAQsurvey}).  Here, 
instead of inverting primes, they work $p$-locally for a fixed prime $p$ and 
invert classes of maps called ``$v_n$-self maps'' (the case of $n = 0$ recovers 
rational homotopy).  Correspondingly, there is the $n$'th unstable monochromatic 
category $M_n^{\,f} {\rm Top}_*$ in the sense of \cite{Bousfield}.  They study 
the functor 
\begin{equation}
 \label{fcr}
 \SS_{T(n)}^{(-)} \co {\rm Ho}(M_n^{\,f} {\rm Top}_*)^{\op} \to {\rm Ho}\big( 
 {\rm Alg}_{\rm Comm}({\rm Sp}_{T(n)}) \big) 
\end{equation}
that sends a space $X$ to the $\SS_{T(n)}$-valued cochains $\SS_{T(n)}^X$.  This 
last spectrum is an algebra for the reduced commutative operad $\rm Comm$ in 
modules over $\SS_{T(n)}$, the localization of the sphere spectrum with respect 
to the telescope of a $v_n$-self map.  

Considering a variant of localization with respect to the Morava $K$-theory 
$K(n)$, Behrens and Rezk have obtained an equivalence 
\[
 \Phi_{K(n)}(X) \xrightarrow{\sim} {\rm TAQ}_{\SS_{K(n)}}(\SS_{K(n)}^X) 
\]
of $K(n)$-local spectra, on a class of spaces $X$ including spheres 
\cite[Theorem 8.1]{BKTAQ} (cf.~\cite[Section 8]{BKTAQsurvey}).  In more detail, 
the left-hand side arises from computing homotopy groups in the source category 
of \eqref{fcr}, where $\Phi_{K(n)} = L_{K(n)} \Phi_n$ is a version of the 
Bousfield-Kuhn functor (cf.~\cite{Kguide}).  This side is a derived realization 
of morphisms in the source.  The right-hand side is the topological 
Andr\'e-Quillen cohomology of $\SS_{K(n)}^X$ as an algebra over the operad 
$\rm Comm$ in $\SS_{K(n)}$-modules.  It is a derived realization of images of 
morphisms under the functor \eqref{fcr} in the target category.  Via a suitable 
Koszul duality between $\rm Comm$ and the Lie operad, we may view the spectrum 
${\rm TAQ}_{\SS_{K(n)}}(\SS_{K(n)}^X)$ as a Lie algebra model for the unstable 
$v_n$-periodic homotopy type of $X$.

\subsection{Main results}

The purpose of this paper is to make available calculations that apply Behrens 
and Rezk's theory to obtain quantitative information about unstable 
$v_n$-periodic homotopy types, in the case of $n = 2$.  These are based on our 
computation of power operations for Morava $E$-theory in \cite{me}.  

Let $E$ be a Morava $E$-theory spectrum of height $2$ with $E_* \cong \BW \cF_p 
\lb a \rb [u^{\pm 1}]$, where $|a| = 0$ and $|u| = -2$.  Recall that the {\em 
completed $E$-homology} functor is defined as $E^\wedge_*(-) \ce 
\pi_*(E \wedge -)_{K(2)}$.  It is $E_0$-linear dual to $E^*(-)$ with more 
convenient properties than $E_*(-)$ (see \cite[Section 3]{cong}).  

Building on and strengthening Rezk's results in \cite[\S 2.13]{h2}, we obtain 
the following.  

\begin{thm}
 \label{thm}
 Given any non-negative integer $m$, denote by $E^\wedge_*(\Phi_2 S^{2 m + 1})$ 
 the completed $E$-homology groups of the Bousfield-Kuhn functor applied to the 
 $(2 m + 1)$-dimensional sphere.  
 \begin{enumerate}[(i)]
  \item The group $E^\wedge_1(\Phi_2 S^{2 m + 1}) \cong 0$ if $m = 0$.  As an 
  $E_0$-module, it equals $(E_0 / p)^{\oplus p - 1}$ if $m = 1$.  It is a 
  quotient of $(E_0 / p^m)^{\oplus p - 1} \oplus E_0 / p^{m - 1}$ if $m > 1$.  
  
  \item More explicitly, 
  \[
   E^\wedge_1(\Phi_2 S^{2 m + 1}) \cong \left\{\!
   \begin{array}{lll}
    \dfrac{\bigoplus_{i = 1}^{p - 1} (E_0 / p^m) \cdot x_i \oplus 
    (E_0 / p^{m - 1}) \cdot x_p}{(r_1, \ldots, r_{m - 1})} & & \text{if \,\! 
    $2 \leq m \leq p + 2$} \\\\
    \dfrac{\bigoplus_{i = 1}^{p - 1} (E_0 / p^m) \cdot x_i \oplus 
    (E_0 / p^{m - 1}) \cdot x_p}{(r_{m - p - 1}, \ldots, r_{m - 1})} & & 
    \text{if \,\! $m > p + 2$} 
   \end{array}
   \right.
  \]
  where $r_j = r_j(x_1, \ldots, x_p) = w_0^{m - 1 - j} \sum_{i = 1}^p 
  d_{i, \, j + 1} \, x_i$.  Here, as in \cite[Theorem 1.6]{me}, 
  \[
   d_{i,\T} = \sum_{n = 0}^{\T - 1} (-1)^{\T - n} \, w_0^n 
              \sum_{\stackrel{\scriptstyle m_1 + \cdots + m_{\T - n} = \T + i} 
              {\stackrel{\scriptstyle 1 \,\leq\, m_s \,\leq\, p + 1}{m_{\T - n} 
              \,\geq\, i + 1}}} w_{m_1} \cdots w_{m_{\T - n}} 
  \]
  where the coefficients $w_i \in E_0 \cong \BW \cF_p \lb a \rb$ are defined by 
  the identity 
  \[
   \sum_{i = 0}^{p + 1} w_i \, b^i = (b - p) \big( b + (-1)^{\,p} \big)^p - 
   \big( a - p^2 + (-1)^{\,p} \big) b 
  \]
  in the variable $b$, so that $w_{p + 1} = 1$, $w_1 = -a$, $w_0 = 
  (-1)^{\,p + 1} p$, and the remaining coefficients 
  \[
   w_i = (-1)^{\,p\,(\,p - i + 1)} \left[ \ch{p}{i - 1} + (-1)^{\,p + 1} \, p \, 
   \ch{\,p\,}{i} \right] 
  \]
  In particular, each relation $r_j$ contains a term $(-1)^{\,j + 1} 
  w_0^{m - 1 - j} w_1^{\,j} \, x_p$.  

  \item The group $E^\wedge_0(\Phi_2 S^{2 m + 1}) \cong 0$ for any $m \geq 0$.  
 \end{enumerate}
\end{thm}

Since $E$ is $2$-periodic, the above determines the completed $E$-homology in 
all degrees.  

\begin{rmk}
 There is a ring structure on $E^\wedge_1(\Phi_2 S^{2 m + 1})$.  Indeed, each 
 generator $x_i$ is a power $b^i$ of a certain element $b$.  See Section 
 \ref{subsec:H^2} for details.  
\end{rmk}

\begin{rmk}
 In \cite[Sections 5.3--5.4]{Wangthesis}, Wang obtained an equivalent 
 presentation of $E^\wedge_*(\Phi_2 S^{2 m + 1})$ for $m = 1$ and any prime $p$.  
 He then used it as the input of a spectral sequence and computed 
 $\pi_*(\Phi_{K(2)} S^3)$ at $p \geq 5$.  
\end{rmk}

\begin{ex}
 We apply Theorem \ref{thm} and compute $E^\wedge_1(\Phi_2 S^{2 m + 1})$ at $p = 
 2$ for small values of $m$.  We have $w(a, b) = b^3 - a b - 2$ so that $w_3 = 
 1$, $w_2 = 0$, $w_1 = -a$, and $w_0 = -2$.  

 \begin{itemize}
  \item When $m = 2$, since $r_1 = 2 x_1 - a x_2$, $E^\wedge_1(\Phi_2 S^5)$ is 
  the quotient of $(E_0 / 4) \cdot x_1 \oplus (E_0 / 2) \cdot x_2$ subject to 
  the relation $a x_2 = 2 x_1$.  
  
  \item When $m = 3$, we have $r_1 = -4 x_1 + 2 a x_2$ and $r_2 = 2 a x_1 - a^2 
  x_2$ so that the relations are 
  \[
   \hskip -1.12cm
   \begin{split}
    a^2 x_2 = & ~ 2 a x_1 \\
    2 a x_2 = & ~ 4 x_1 
   \end{split}
  \]
  
  \item When $m = 4$, we have $r_1 = 8 x_1 - 4 a x_2$, $r_2 = -4 a x_1 + 2 a^2 
  x_2$, and $r_3 = 2 a^2 x_1 + (-a^3 + 4) x_2$.  Thus the relations are 
  \[
   \hskip -.12cm 
   \begin{split}
      a^3 x_2 = & ~ 2 a^2 x_1 + 4 x_2 \\
    2 a^2 x_2 = & ~ 4 a x_1 \\
      4 a x_2 = & ~ 8 x_1 
   \end{split}
  \]
  
  \item When $m = 5 > p + 2$, we have $r_2 = 8 a x_1 - 4 a^2 x_2$, $r_3 = -4 a^2 
  x_1 + (2 a^3 - 8) x_2$, and $r_4 = (2 a^3 - 8) x_1 + (-a^4 + 8 a) x_2$.  Thus 
  the relations are 
  \[
   \hskip .5cm 
   \begin{split}
      a^4 x_2 = & ~ (2 a^3 - 8) x_1 + 8 a x_2 \\
    2 a^3 x_2 = & ~ 4 a^2 x_1 + 8 x_2 \\
    4 a^2 x_2 = & ~ 8 a x_1 
   \end{split}
  \]
 \end{itemize}

 The relations above show that the bounds for $p$-power torsion in Theorem 
 \ref{thm} are sharp (see \cite[\S 2.5]{Btwo} and \cite{Selick}).  Also, as in 
 part (ii) of the theorem, each $r_j$ contains a term $2^{m - 1 - j} a^{\,j} 
 x_2$.  Unfortunately, it is  impossible to simplify the relations for 
 $E^\wedge_1(\Phi_2 S^{2 m + 1})$ into $2^{m - 1 - j} a^{\,j} x_2 = 0$ by an 
 $E_0$-linear change of variables with $x_i$.  See Remark \ref{rmk} below.  
\end{ex}

\subsection{A comparison to the case of $n = 1$}

As an application of Behrens and Rezk's theory, Theorem \ref{thm} is a step 
toward the program initiated in \cite{AM} to compute the unstable 
$v_n$-periodic homotopy groups of spheres using stable $v_n$-periodic homotopy 
groups and Goodwillie calculus.  See also \cite{Wang, Wangthesis}.  Given the 
computations of Davis and Mahowald in the 1980s for the case of $n = 1$, we 
discuss a version of Theorem \ref{thm} at height $1$ according to this program.  

Davis and Mahowald showed that, $K(1)$-locally at a fixed prime $p$, the Moore 
spectrum $\SS^{-1} / p^m$ with $i$'th space $S^{i - 1} \cup_{p^m} e^i$ is 
equivalent to the suspension spectrum of a stunted $B\Sigma_p$.  Via the 
Goodwillie tower of the identity functor on the category of pointed spaces, the 
latter can be identified with $\Phi_1(S^{2 m + 1})$, again $K(1)$-locally (or 
$T(1)$-locally, due to the validity of the Telescope Conjecture at height $1$).  
We thus obtain a variant of Theorem \ref{thm}.  

\begin{prop}
 \label{prop}
 Let $E$ be a Morava $E$-theory spectrum of height $1$, with $E_0 \cong \BW 
 \cF_p$.  Assume that $p \neq 3$ and that, if $p = 2$, $m \equiv 0, 3 \md 4$.  
 Then 
 \[
  E^\wedge_0(\Phi_1 S^{2 m + 1}) \cong E_0 / p^m \hskip 2cm m \geq 0 
 \]
\end{prop}
\begin{proof}
 For non-negative integers $k$ and $b$, let $L(k)_b \ce e_k \Sigma^\infty 
 (B\BF_p^k)^{b \bar{\rho}_k}$ be the stable summand of the Thom space 
 $(B\BF_p^k)^{b \bar{\rho}_k}$ associated to the Steinberg idempotent, where 
 $b \bar{\rho}_k$ denotes the direct sum of $b$ copies of the $k$-dimensional 
 reduced real regular representation of $\BF_p^k$ (cf.~\cite[Remark 
 5.4]{BKTAQ}).  For $b \leq t$, let $L(k)_b^t$ be the fiber of the natural map 
 of spectra $L(k)_b \to L(k)_{t + 1}$ (see \cite[Chapter 2, esp.~Section 
 2.3]{GoodEHP}).  When $k = 1$, it is a stunted $B\Sigma_p$ 
 (cf.~\cite{MitchellPriddy}).  In particular, if $p = 2$, $L(1)_b^t \simeq 
 \PP_b^t$, the suspension spectrum of the stunted real projective space 
 $\BR\BP^t / \BR\BP^{b - 1}$.  In this case, when $m = 4 n$, we have 
 \[
  \begin{split}
   \Phi_1(S^{2 m + 1}) \simeq & ~ L_{K(1)} \Sigma^{2 m + 1} L(1)_1^{2 m} 
   \hskip .48cm \text{by \cite[Theorem 7.20]{Koverview}} \\
                       \simeq & ~ L_{K(1)} \Sigma^{2 m + 1} \PP_1^{8 n} \\
                       \simeq & ~ L_{K(1)} \Sigma^{2 m + 1} \PP_{1 - 8 n}^0 
   \hskip .6cm \text{by \cite[Proposition 2.1]{DM}} \\
                       \simeq & ~ L_{K(1)} \Sigma^{2 m + 1} \SS^{-1} / 2^m
   \hskip .31cm \text{by \cite[proof of Theorem 4.2]{DM}} \\
                       \simeq & ~ L_{K(1)} \SS^{2 m} / 2^m
  \end{split}
 \]
 and thus $E^\wedge_0(\Phi_1 S^{2 m + 1}) \cong E_0 / 2^m$.  It is similar when 
 $m = 4 n + 3$.  For $p > 3$, we apply \cite[Corollary 1.7 and Theorem 
 1.8]{Davis}.  
\end{proof}

\begin{rmk}
 There has been extensive work on the case of $v_1$-periodic homotopy theory.  
 See \cite[\S 9.11]{Bodd} and \cite[\S\S 8.6--8.7]{Btwo} for an alternative 
 approach to a more general result than the above where the assumptions on $p$ 
 and $m$ are removed (cf.~\cite[\S 2.13]{h2}).  See also \cite[esp.~Theorem 
 3.1]{Dsurvey} and the references therein for related information in this case.  
\end{rmk}

\begin{rmk}
 For any $n$, since $\Phi_n$ preserves fiber sequences, there is a natural map 
 \[
  \Sigma^2 \Phi_n(X) \to \Phi_n(\Sigma^2 X) 
 \]
 In view of the $2$-periodicity of $E$, this induces a map on completed 
 $E$-homology in the same degree.  Thus the groups 
 $\{E^\wedge_*(\Phi_n S^{2 m + 1})\}_{m \geq 0}$ form a direct system.  Homotopy 
 (co)limits of generalized Moore spectra are closely related to various kinds of 
 localizations of the sphere spectrum (see, e.g., \cite[Proposition A.3]{AM} and 
 \cite[Proposition 7.10]{HS99}).  On the other hand, note that completed 
 $E$-homology does not preserve homotopy colimits \cite{Hovey}.  Nevertheless, 
 based on computational evidence from Theorem \ref{thm}, Proposition \ref{prop}, 
 and further, we hope to study the relationship between the $K(n)$-local sphere 
 and the Bousfield-Kuhn functor on odd-dimensional spheres hinted in \cite[\S\S 
 3.20--3.21]{fkl'}.  
\end{rmk}

\subsection{Acknowledgements}

I thank Mark Behrens, Paul Goerss, Guchuan Li, Charles Rezk, Guozhen Wang, and 
Zhouli Xu for helpful discussions.  I am especially grateful to Mark for 
introducing me into the program of computing unstable periodic homotopy groups 
of spheres and to Charles for sharing his knowledge in power operations.  

I thank Lennart Meier for his helpful remarks on an earlier draft of this paper, 
particularly one that saved me from an error in the main results.

\section{Koszul complexes for modules over the Dyer-Lashof algebra of Morava 
$E$-theory}

Let $E$ be a Morava $E$-theory spectrum of height $n$ at the prime $p$.  Its 
formal group $\Spf E^0 \BC \BP^\infty$ over $E_0 \cong \BW \cF_p \lb u_1, 
\ldots, u_{n - 1} \rb$ is the Lubin-Tate universal deformation of a formal group 
$\BG$ over $\cF_p$ of height $n$.  

Generalizing the Lubin-Tate deformation theory, Strickland shows that for each 
$k \geq 0$ there is a ring $A_k \cong E^0 B\Sigma_{p^k} / I_k$ classifying 
subgroups of degree $p^k$ in the universal deformation, where $I_k$ is the ideal 
generated by the image of all transfer maps from inclusions of the form 
$\Sigma_i \times \Sigma_{p^k - i} \subset \Sigma_{p^k}$ with $0 < i < p^k$ 
\cite[Theorem 1.1]{Str98}.  In particular, $A_0 \cong E_0$ and there are ring 
homomorphisms 
\[
 s = s_k, \, t = t_k \co A_0 \to A_k \quad \ad \quad \mu = \mu_{k, m} \co 
 A_{k + m} \to A_k \! \tensor[^s]{\otimes}{^t_{\!A_0}} A_m 
\]
classifying the source and target of an isogeny of degree $p^k$ on the universal 
deformation and the composition of two isogenies.  

As $E$ is an $E_\infty$-ring spectrum, there are (additive) power operations 
acting on the homotopy of $K(n)$-local commutative $E$-algebra spectra.  A {\em 
$\G$-module} is an $A_0$-module $M$ equipped with structure maps (the power 
operations) 
\[
 P_k \co M \to ~^t\,A_k \! \tensor[^s]{\otimes}{_{A_0}} M \hskip 1cm k \geq 0 
\]
which are a compatible family of $A_0$-module homomorphisms.  These power 
operations form the {\em \DL algebra} $\G$ for the $E$-theory, with graded 
pieces $\G[k] \ce \Hom_{A_0}(^sA_k, A_0)$, $k \geq 0$.  There is a tensor 
product $\otimes$ for $\G$-modules \cite[\S 4.1]{h2}.  

The structure of a $\G$-module is determined by $P_1$, subject to a condition 
involving $A_2$, i.e.~the existence of the dashed arrow in the diagram  
\begin{equation}
 \label{diag}
 \begin{tikzpicture}[baseline=(current bounding box.center)]
  \node (LT) at (0, 2.5) {$M$};
  \node (RT) at (4, 2.5) {$^t\,A_1 \! \tensor[^s]{\otimes}{_{A_0}} M$};
  \node (LB) at (0, 0) {$^t\,A_2 \! \tensor[^s]{\otimes}{_{A_0}} M$};
  \node (RB) at (4, 0) {$^t\,A_1 \! \tensor[^s]{\otimes}{^t_{\!A_0}} 
                             A_1 \! \tensor[^s]{\otimes}{_{A_0}} M$};
  \draw [->] (LT) -- node [above] {$\scriptstyle P_1$} (RT);
  \draw [dashed, ->] (LT) -- (LB);
  \draw [->] (RT) -- node [right] {$\scriptstyle \id \otimes P_1$} (RB);
  \draw [->] (LB) -- node [above] {$\scriptstyle \mu \otimes \id$} (RB);
 \end{tikzpicture}
\end{equation}
\cite[Proposition 7.2]{h2}.  This manifests the fact that the ring $\G$ is {\em 
Koszul} and, in particular, {\em quadratic} \cite{koszul}.  

Let $D_0 \ce A_0$, $D_1 \ce A_1$, and 
\[
 D_k \ce \coker \left( \bigoplus_{i = 0}^{k - 2} A_1^{\otimes i} \otimes A_2 
 \otimes A_1^{k - i - 2} \xrightarrow{\id \otimes \mu \otimes \id} 
 A_1^{\otimes k} \right) \hskip 1cm k \geq 2 
\]
Given $\G$-modules $M$ and $N$, Rezk defines the {\em Koszul complex} 
$\CC^\bullet(M, N)$ by 
\[
 \CC^k(M, N) \ce \Hom_{A_0}(M, D_k \otimes_{A_0} N)
\]
with appropriate coboundary maps \cite[\S 7.3]{h2}.  

\begin{prop}
 \label{prop:koszul}
 If $M$ is projective as an $A_0$-module, then 
 \[
  \Ext_\G^k(M, N) \cong H^k \CC^\bullet(M, N) 
 \]
 In particular, if $k > n$, $D_k \cong 0$ and so $\Ext_\G^k(M, N) \cong 0$.  
\end{prop}
\begin{proof}
 This is {\cite[Proposition 7.4]{h2}}.  
\end{proof}

\subsection{The case of $n = 2$}
\label{subsec:n = 2}

Choose a {\em preferred $\CP_N$-model} for $E$ in the sense of \cite[Definition 
3.29]{ho} so that the formal group of $E$ is isomorphic to the formal group of a 
universal deformation of a supersingular elliptic curve satisfying a list of 
properties.  

Using the theory of dual isogenies of elliptic curves, Rezk identifies that $D_2 
\cong A_1 / s(A_0)$ \cite[Proposition 9.3]{h2}.  He also classifies $\G$-modules 
of rank $1$ in this case \cite[Proposition 9.7]{h2}.  In particular, each of 
them takes the form $L_\B$ with structure map 
\[
 P \co L_\B \to ~^t \, A_1 \! \tensor[^s]{\otimes}{_{A_0}} L_\B 
\]
\[
 x \mapsto \B \otimes x ~~ 
\]
where $x$ is a generator for the underlying $A_0$-module, and $\B \in A_1$ is 
such that $\iota(\B) \cdot \B \in s(A_0)$ with $\iota(-)$ the Atkin-Lehner 
involution (this condition on $\B$ corresponds to the condition in 
\eqref{diag}).  Moreover, $L_1$ is the unit object in the symmetric monoidal 
category of $\G$-modules with respect to $\otimes$ and $L_{\B_1} \otimes 
L_{\B_2} \cong L_{\B_1 \B_2}$.  Thus $L_\B$ is $\otimes$-invertible as a 
$\G$-module if and only if $\B \in A_1^\times$.  

Now let $M = L_\A$ and $N = L_\B$.  We have identifications 
\[
 A_0 \xrightarrow{\sim} \CC^0(M, N) = \Hom_{A_0}(M, N) \hskip 5.6cm f \mapsto (x 
 \mapsto f \, y) 
\]
\[
 A_1 \xrightarrow{\sim} \CC^1(M, N) = \Hom_{A_0}(M, ~ ^t \, A_1 \! 
 \tensor[^s]{\otimes}{_{A_0}} N) \hskip 3.65cm g \mapsto (x \mapsto g \otimes y) 
\]
\[
 A_1 / s(A_0) \xrightarrow{\sim} \CC^2(M, N) = \Hom_{A_0}\left( M, ~ 
 ^{\iota^2 s} \, \big( A_1 / s(A_0) \big) \! \tensor[^s]{\otimes}{_{A_0}} N 
 \right) \hskip .75cm h \mapsto (x \mapsto h \otimes y) 
\]
Thus the Koszul complex in this case is 
\[
 A_0 \xrightarrow{d_0} A_1 \xrightarrow{d_1} A_1 / s(A_0) 
\]
with $d_0 \, f = \iota(f) \B - f \A$ and $d_1 g = \iota(g) \B + g \iota(\A)$ 
\cite[\S 9.18]{h2}.  

More explicitly, we have identifications 
\[
 A_0 \cong \BW \cF_p \lb a \rb \quad \ad \quad A_1 \cong \BW \cF_p \lb a, b \rb 
 / \big( w(a, b) \big) 
\]
where 
\[
 w(a, b) = \sum_{i = 0}^{p + 1} w_i \, b^i = (b - p) \big( b + (-1)^{\,p} 
 \big)^p - \big( a - p^2 + (-1)^{\,p} \big) b 
\]
\cite[Theorem 1.2]{me}.  Note that the parameters $a$ and $b$ are chosen as in 
\cite[\S 9.15]{h2} and they correspond precisely to $h$ and $\A$ in \cite{ho, 
me}.  In particular, the $\G$-module of invariant 1-forms is $\o = L_b$.  

\begin{rmk}
 \label{rmk}
 As we will see in Section \ref{sec:pf}, the generators $x_i$ in Theorem 
 \ref{thm} (ii) depend on the choice of the parameter $b$ for $A_1$.  We do not 
 know if a different choice would make the presentations simpler.  
\end{rmk}

The ring homomorphism $s \co A_0 \to A_1$ is simply the inclusion of scalars, as 
$A_1$ is a free left module over $A_0$ of rank $p + 1$.  We will thus abbreviate 
$s(A_0)$ as $A_0$.  Following \cite{h2}, we will also abbreviate $\iota(x)$ as 
$x'$, which is written as $\widetilde{x}$ in \cite{ho, me}.  Note that 
$w_{p + 1} = 1$, $p | w_i$ for $2 \leq i \leq p$, $w_1 = -a$, and 
\begin{equation}
 \label{w_0}
 w_0 = (-1)^{p + 1} p = b b' 
\end{equation}
\cite[(3.30)]{ho}.

\section{Computing with Koszul complexes}
\label{sec:ext}

Recall that $\o = L_b$ is the $\G$-module of invariant 1-forms defined in 
Section \ref{subsec:n = 2}.  Write $\nul \ce L_0$, the $\G$-module annihilated 
by $\G$.  In this section, we compute $\Ext_\G^*(\o^m, \nul)$ for $m \geq 0$.  
By Proposition \ref{prop:koszul}, 
\[
 \Ext_\G^*(\o^m, \nul) \cong H^*\CC^\bullet(L_{b^m}, L_0) 
\]
where 
\[
 \CC^\bullet(L_{b^m}, L_0) \co A_0 \xrightarrow{-b^m} A_1 \xrightarrow{b'^m} A_1 
 / A_0 
\]

\begin{prop}
 For all $m \geq 0$, $H^0\CC^\bullet(L_{b^m}, L_0) \cong 0$.  
\end{prop}
\begin{proof}
 We need to show that $A_0 \xrightarrow{-b^m} A_1$ is injective.  Given $f(a) 
 \in A_0 \cong \BW \cF_p \lb a \rb$, suppose $-b^m \cdot f(a) = 0 \in A_1 \cong 
 \BW \cF_p \lb a, b \rb / \big(w(a, b)\big)$.  
 
 If $0 \leq m \leq p$, since $w(a, b)$ is a polynomial in $b$ of degree $p + 1$ 
 with coefficients in $A_0$, clearly $f(a)$ must be $0$.  
 
 If $m > p$, we need only show that $b^m \not\equiv 0 \md w(a, b)$.  Since 
 $w(a, b) \equiv b (b^p - a) \md p$, we have $b^{p + 1} \equiv a b \md (p, w)$, 
 and thus $b^m \not\equiv 0 \md (p,w)$.  
\end{proof}

\begin{prop}
 For all $m \geq 0$, $H^1\CC^\bullet(L_{b^m}, L_0) \cong 0$.  
\end{prop}
\begin{proof}
 Let $g(a, b)$ be a polynomial in $b$ of degree at most $p$ with coefficients in 
 $A_0$ that represents an element in $A_1$.  Suppose $b'^m \cdot g(a, b) = 0 \in 
 A_1 / A_0$.  We need to show that $g(a, b) \equiv -b^m \cdot f(a) \md w(a, b)$ 
 for some $f(a) \in A_0$.  
 
 We do this by induction on $m$.  The case of $m = 0$ is clear.  Let $m \geq 1$.  
 By the induction hypothesis, since $b'^{m - 1} \cdot b' g(a, b) = 0 \in A_1 / 
 A_0$, we have $b' g(a, b) \equiv -b^{m - 1} \cdot f(a) \md w(a, b)$.  
 Multiplying both sides by $b$, in view of \eqref{w_0}, we get 
 \begin{equation}
  \label{cong1}
  w_0 g(a, b) \equiv -b^m f(a) \md w 
 \end{equation}
 and thus 
 \begin{equation}
  \label{cong2}
  0 \equiv -b^m f(a) \md (p, w) 
 \end{equation}
 Since $b^{p + 1} \equiv a b \md (p, w)$, \eqref{cong2} implies that $p \,|\,\, 
 f(a)$ in $A_1$.  As $p$ is not a zero-divisor in $A_1$, \eqref{cong1} implies 
 that $g(a, b) \equiv -b^m \tilde{f}(a) \md w$ for some $\tilde{f}(a) \in A_0$.  
\end{proof}

\subsection{The second cohomology}
\label{subsec:H^2}

Finally, we compute $H^2\CC^\bullet(L_{b^m}, L_0)$.  Write $B_m \ce 
H^2\CC^\bullet(L_{b^m}, L_0) \cong A_1 / (A_0 + b'^m A_1)$.  Clearly, $B_0 \cong 
0$.  Let $m > 0$ for the rest of this section.  

As a free module over $A_0$, the ring $A_1$ has a basis consisting of 
\begin{equation}
 \label{basis1}
 1, \, b, \, b^2, \, \ldots, \, b^p 
\end{equation}

\begin{prop}
 In the $A_0$-module $B_m$, $p^m b^i = 0$ for $1 \leq i \leq p - 1$ and 
 $p^{m - 1} b^p = 0$.  
\end{prop}
\begin{proof}
 In view of \eqref{w_0}, we have $w_0^m b^i = b'^m b^m b^i = b'^m b^{m + i} = 0$ 
 and 
 \begin{equation}
  \label{equiv}
  \begin{split}
   w_0^{m - 1} b^p = & ~ w_0^{m - 1} (-b' - w_p b^{p - 1} - \cdots - w_2 b) \\
                   = & ~ -b^{m - 1} b'^m - w_0^{m - 1} w_p b^{p - 1} - \cdots - 
                         w_0^{m - 1} w_2 b \\
                   = & ~ -w_0^{m - 1} w_p b^{p - 1} - \cdots - w_0^{m - 1} w_2 b 
  \end{split}
 \end{equation}
 Since $p | w_i$ for $2 \leq i \leq p$, the last expression has a factor of 
 $w_0^m$ and so vanishes as we have just shown.  
\end{proof}

Let $1 \leq m \leq p$.  Under the map of multiplication by $b'^m$, the elements 
in \eqref{basis1} become 
\begin{equation}
 \label{basis2}
 b'^m, \, w_0 b'^{m - 1}, \, w_0^2 b'^{m - 2}, \, \ldots, \, w_0^{m - 1} b', \, 
 w_0^m, \, w_0^m b, \, \ldots, \, w_0^m b^{p - m} 
\end{equation}
Note that $w_0^{m - 1} b' = 0$ in $B_m$ is equivalent to \eqref{equiv}.  Thus, 
as a quotient of $(A_0 / p^m)^{\oplus p - 1} \oplus A_0 / p^{m - 1}$ from the 
above proposition, $B_m$ has relations given precisely by the vanishing of the 
first $(m - 1)$ terms in \eqref{basis2}.  

To write down these relations explicitly, with notation as in \cite[Theorem 
1.6~(ii)]{me}, we have 
\[
 b'^k = d_{p, k} b^p + d_{p - 1, k} b^{p - 1} + \cdots + d_{0, k} \hskip 1cm 2 
 \leq k \leq m \leq p 
\]
(cf.~\cite[Section 4.1]{me} for $k > p$).  In particular, the formula for the 
coefficient $d_{p, k}$ has a leading term $(-1)^k w_1^{k - 1} w_{p + 1}$.  Thus 
setting $w_0^{m - k} b'^k$ to be zero in $B_m$ gives an expression for 
$w_0^{m - k} w_1^{k - 1} b^p$ in terms of an $A_0$-linear combination of 
$b^{p - 1}$, \ldots, $b$, and, possibly, $b^p$ itself if there are more than one 
term in $d_{p, k}$ not divisible by $p^{m - 1}$.  

The case of $m > p$ is similar.

\section{Proof of Theorem 1.2}
\label{sec:pf}

Recall that given a Morava $E$-theory $E$ of height $n$, the completed 
$E$-homology functor is defined as $E^\wedge_*(-) \ce \pi_*(E \wedge -)_{K(n)}$.  
In particular, 
\begin{equation}
 \label{ceh}
 E^\wedge_*(\Phi_n X) \cong E^\wedge_*(\Phi_{K(n)} X) 
\end{equation}
since the map $\id \wedge L_{K(n)} \co E \wedge \Phi_n X \to E \wedge L_{K(n)} 
\Phi_n X$ induces a $K(n)$-equivalence by the K\"unneth isomorphism.  

In \cite{h2}, Rezk sets up a composite functor spectral sequence (CFSS) followed 
by a mapping space spectral sequence (MSSS) to compute the homotopy groups of 
derived mapping spaces $\widehat{\CR}_E(A, B)$ between $K(n)$-local augmented 
commutative $E$-algebras $A$ and $B$.  He identifies the $E_2$-term in the CFSS 
as $\Ext$-groups over the \DL algebra $\G$.  The CFSS converges to the 
$E_2$-term in the MSSS.  

In particular, \cite[\S 2.13]{h2} shows that this setup specializes to compute 
the $E$-cohomology of the topological Andr\'e-Quillen homology 
${\rm TAQ}^{\SS_{K(n)}}(\SS_{K(n)}^{S^{2 m + 1}_+})$, and that the two spectral 
sequences both collapse at the $E_2$-term when $n = 2$.  Here $A = 
E^{S^{2 m + 1}_+} \ce F(\Sigma_+^\infty S^{2 m + 1}, E)$ and $B = E \rtimes E$ 
is a {\em square-zero extension} (see \cite[\S 5.10]{h2}).  

Now, by \cite[Theorem 8.1]{BKTAQ} and \eqref{ceh}, we identify the abutment of 
the MSSS as 
\[
 \pi_{t - s} \widehat{\CR}_E(E^{S^{2 m + 1}_+}, E \rtimes E) \cong \pi_{t - s} F 
 \big( {\rm TAQ}^{\SS_{K(2)}}(\SS_{K(2)}^{S^{2 m + 1}_+}), E \big) \cong 
 E^\wedge_{t - s}(\Phi_2 S^{2 m + 1}) 
\]
For a fixed $t$, Rezk identifies the possibly nonzero terms on the $E_2$-page of 
the CFSS as $\Ext_{{\rm Mod}_\G^\star}^s(\o^m, \o^{(t - 1) / 2} \otimes \nul)$, 
where ${\rm Mod}_\G^\star$ is the category of {\em $\BZ/2$-graded} $\G$-modules 
in the sense of \cite[\S 5.6]{h2}.  Thus for degree and periodicity reasons, we 
may set $t = 1$ and the calculations in Section \ref{sec:ext} then complete the 
proof, with $b^i$ written as $x_i$ in Theorem \ref{thm}.  \qed

% \bibliographystyle{amsalpha}
% \bibliography{bkos}

\vspace{-.63in}
\begin{thebibliography}

\section*{\leftskip=-.44in References \vspace{.17in}}

\bibitem[Arone-Ching2015]{AroneChing}
Gregory Arone and Michael Ching, \emph{A classification of {T}aylor towers of
  functors of spaces and spectra}, Adv. Math. \textbf{272} (2015), 471--552.
  \MRn{3303239}{}

\bibitem[Arone-Mahowald1999]{AM}
Greg Arone and Mark Mahowald, \emph{The {G}oodwillie tower of the identity
  functor and the unstable periodic homotopy of spheres}, Invent. Math.
  \textbf{135} (1999), no.~3, 743--788. \MRn{1669268}{}

\bibitem[Behrens2012]{GoodEHP}
Mark Behrens, \emph{The {G}oodwillie tower and the {EHP} sequence}, Mem. Amer.
  Math. Soc. \textbf{218} (2012), no.~1026, xii+90. \MRn{2976788}{}

\bibitem[Behrens-Rezk2015]{BKTAQ}
Mark Behrens and Charles Rezk, \emph{The Bousfield-Kuhn functor and topological 
  Andr\'e-Quillen cohomology}, available at \\
  \href{http://www3.nd.edu/~mbehren1/papers/BKTAQ6.pdf}
  {http://www3.nd.edu/\textasciitilde mbehren1/papers/BKTAQ6.pdf}.

\bibitem[Behrens-Rezk2016]{BKTAQsurvey}
Mark Behrens and Charles Rezk, \emph{Spectral algebra models of unstable 
  $v_n$-periodic homotopy theory}, available at \\
  \href{http://www3.nd.edu/~mbehren1/papers/BKTAQsurvey2.pdf}
  {http://www3.nd.edu/\textasciitilde mbehren1/papers/BKTAQsurvey2.pdf}.

\bibitem[Bousfield1999]{Bodd}
A.~K. Bousfield, \emph{The {$K$}-theory localizations and {$v\sb 1$}-periodic
  homotopy groups of {$H$}-spaces}, Topology \textbf{38} (1999), no.~6,
  1239--1264. \MRn{1690156}{}

\bibitem[Bousfield2001]{Bousfield}
A.~K. Bousfield, \emph{On the telescopic homotopy theory of spaces}, Trans.
  Amer. Math. Soc. \textbf{353} (2001), no.~6, 2391--2426 (electronic).
  \MRn{1814075}{}

\bibitem[Bousfield2005]{Btwo}
A.~K. Bousfield, \emph{On the 2-primary {$v\sb 1$}-periodic homotopy groups of 
  spaces}, Topology \textbf{44} (2005), no.~2, 381--413. \MRn{2114954}{}

\bibitem[Davis1986]{Davis}
Donald~M. Davis, \emph{Odd primary {$b{\rm o}$}-resolutions and {$K$}-theory
  localization}, Illinois J. Math. \textbf{30} (1986), no.~1, 79--100.
  \MRn{822385}{}

\bibitem[Davis1995]{Dsurvey}
Donald~M. Davis, \emph{Computing {$v\sb 1$}-periodic homotopy groups of spheres 
  and some compact {L}ie groups}, Handbook of algebraic topology, North-Holland,
  Amsterdam, 1995, pp.~993--1048. \MRn{1361905}{}

\bibitem[Davis-Mahowald1987]{DM}
Donald~M. Davis and Mark Mahowald, \emph{Homotopy groups of some mapping
  telescopes}, Algebraic topology and algebraic {$K$}-theory ({P}rinceton,
  {N}.{J}., 1983), Ann. of Math. Stud., vol. 113, Princeton Univ. Press,
  Princeton, NJ, 1987, pp.~126--151. \MRn{921475}{}

\bibitem[Goerss1995]{Goerss}
Paul~G. Goerss, \emph{Simplicial chains over a field and {$p$}-local homotopy
  theory}, Math. Z. \textbf{220} (1995), no.~4, 523--544. \MRn{1363853}{}

\bibitem[Heuts2016]{Heuts}
Gijs Heuts, \emph{Goodwillie approximations to higher categories}. 
\AX{1510.03304}

\bibitem[Hovey2008]{Hovey}
Mark Hovey, \emph{Morava {$E$}-theory of filtered colimits}, Trans. Amer. Math.
  Soc. \textbf{360} (2008), no.~1, 369--382 (electronic). 
  \MRn{2342007}{(2008g:55007)}

\bibitem[Hovey-Strickland1999]{HS99}
Mark Hovey and Neil~P. Strickland, \emph{Morava {$K$}-theories and
  localisation}, Mem. Amer. Math. Soc. \textbf{139} (1999), no.~666, viii+100.
  \MRn{1601906}{(99b:55017)}

\bibitem[K{\v{r}}{\'{\i}}{\v{z}}1993]{Kriz}
Igor K{\v{r}}{\'{\i}}{\v{z}}, \emph{{$p$}-adic homotopy theory}, Topology Appl.
  \textbf{52} (1993), no.~3, 279--308. \MRn{1243609}{}

\bibitem[Kuhn2007]{Koverview}
Nicholas~J. Kuhn, \emph{Goodwillie towers and chromatic homotopy: an overview},
  Proceedings of the {N}ishida {F}est ({K}inosaki 2003), Geom. Topol. Monogr.,
  vol.~10, Geom. Topol. Publ., Coventry, 2007, pp.~245--279. \MRn{2402789}{}

\bibitem[Kuhn2008]{Kguide}
Nicholas~J. Kuhn, \emph{A guide to telescopic functors}, Homology, Homotopy
  Appl. \textbf{10} (2008), no.~3, 291--319. \MRn{2475626}{}

\bibitem[Mandell2001]{Mandell}
Michael~A. Mandell, \emph{{$E\sb \infty$} algebras and {$p$}-adic homotopy
  theory}, Topology \textbf{40} (2001), no.~1, 43--94. \MRn{1791268}{}

\bibitem[Mitchell-Priddy1983]{MitchellPriddy}
Stephen~A. Mitchell and Stewart~B. Priddy, \emph{Stable splittings derived from
  the {S}teinberg module}, Topology \textbf{22} (1983), no.~3, 285--298.
  \MRn{710102}{}

\bibitem[Quillen1969]{Quillen}
Daniel Quillen, \emph{Rational homotopy theory}, Ann. of Math. (2) \textbf{90}
  (1969), 205--295. \MRn{0258031}{}

\bibitem[Rezk2009]{cong}
Charles Rezk, \emph{The congruence criterion for power operations in {M}orava
  {$E$}-theory}, Homology, Homotopy Appl. \textbf{11} (2009), no.~2, 327--379.
  \MRn{2591924}{(2011e:55021)}

\bibitem[Rezk2012]{koszul}
Charles Rezk, \emph{Rings of power operations for {M}orava {$E$}-theories are
  {K}oszul}. \AX{1204.4831}

\bibitem[Rezk2013]{h2}
Charles Rezk, \emph{Power operations in Morava E-theory: structure and 
  calculations}, available at 
  \href{http://www.math.uiuc.edu/~rezk/power-ops-ht-2.pdf}
  {http://www.math.uiuc.edu/\textasciitilde rezk/power-ops-ht-2.pdf}.

\bibitem[Rezk2016]{fkl'}
Charles Rezk, \emph{Elliptic cohomology and elliptic curves}, Felix Klein 
  Lectures, Bonn 2015, available at 
  \href{http://www.math.uiuc.edu/~rezk/felix-klein-lectures-notes.pdf}
  {http://www.math.uiuc.edu/\textasciitilde 
  rezk/felix-klein-lectures-notes.pdf}.

\bibitem[Selick1988]{Selick}
Paul Selick, \emph{Moore conjectures}, Algebraic topology---rational homotopy
  ({L}ouvain-la-{N}euve, 1986), Lecture Notes in Math., vol. 1318, Springer,
  Berlin, 1988, pp.~219--227. \MRn{952582}{}

\bibitem[Strickland1998]{Str98}
N.~P. Strickland, \emph{Morava {$E$}-theory of symmetric groups}, Topology
  \textbf{37} (1998), no.~4, 757--779. \MRn{1607736}{(99e:55008)}

\bibitem[Wang2014]{Wang}
Guozhen Wang, \emph{The monochromatic Hopf invariant}. \AX{1410.7292}

\bibitem[Wang2015]{Wangthesis}
Guozhen Wang, \emph{Unstable chromatic homotopy theory}, ProQuest LLC, Ann
  Arbor, MI, 2015, Thesis (Ph.D.)--Massachusetts Institute of Technology.
  \MRn{3427198}{}

\bibitem[Zhu2015a]{ho}
Yifei Zhu, \emph{The {H}ecke algebra action on {M}orava {$E$}-theory of height
  $2$}, available at \href{https://yifeizhu.github.io/ho.pdf}
  {https://yifeizhu.github.io/ho.pdf}.

\bibitem[Zhu2015b]{me}
Yifei Zhu, \emph{Modular equations for {L}ubin-{T}ate formal groups at chromatic
  level $2$}, available at \href{https://yifeizhu.github.io/me.pdf}
  {https://yifeizhu.github.io/me.pdf}.

\end{thebibliography}
% \end{document}
\renewcommand\refname{}
\newcommand{\AX}[1]{\href{http://arxiv.org/abs/#1}{arXiv:#1}}
\newcommand{\MRn}[2]{\href{http://www.ams.org/mathscinet-getitem?mr=#1}{MR#1#2}}
\newcommand{\name}{TateNormalLevelResolutions.pdf}
\wt{.}

\end{document}